\documentclass[10pt,a4paper]{amsart}

\usepackage{amssymb,amsbsy,amsmath,amsfonts,amssymb,amscd, mathrsfs}

\usepackage{enumerate}

\usepackage{color}
\usepackage[english]{babel}
\usepackage[T1]{fontenc}
\usepackage{indentfirst}
\makeatletter
\@addtoreset{equation}{section}
\makeatother

\newtheorem{statement}{}[section]
\newtheorem{definition}[statement]{Definition}
\newtheorem{theorem}[statement]{Theorem}

\newtheorem{corol}[statement]{Corollary}
\newtheorem{lemma}[statement]{Lemma}

\def\span{\text{\rm span}}

\newcommand{\C}{\mathbb C}

\newcommand{\R}{\mathbb R}

\newcommand\e{{\rm e}}
\newcommand{\eps}{\varepsilon}
\newcommand\ind{{\rm 1\kern-.30em I}}

\newcommand\dis{\displaystyle}

\newcommand{\biindice}[3]%
{%

\begin{array}[t]{c}
{\displaystyle #1}\\
{\scriptstyle #2}\\
{\scriptstyle #3}
\end{array}

}

\begin{document}

\title{\bf $L^2$-M\"untz spaces as model spaces.}

\author[Fricain]{Emmanuel Fricain}
\address{Laboratoire Paul Painlev\'e, Universit\'e Lille 1, 59 655 Villeneuve d'Ascq, France }
\email{emmanuel.fricain@math.univ-lille1.fr}

\author[Lef\`evre]{Pascal Lef\`evre}
\address{Laboratoire de Math\'ematiques de Lens, Universit\'e d'Artois, Rue Jean Souvraz, 62307 Lens, France}
\email{pascal.lefevre@univ-artois.fr}

\date{ \footnotesize \today}

\maketitle

\bigskip

 {\bf Abstract.} {\it We emphasize a bridge between two areas of function theory: hilbertian M\"untz spaces and model spaces of the Hardy space of the right half plane. We give miscellaneous applications of this viewpoint to hilbertian M\"untz spaces.}
\medskip

\par\medskip


\section{Introduction} 

For quite a long time, many mathematicians paid a lot of attention to the theory of model spaces: there is a wide literature on the subject (see for instance Nikolski treatise \cite{N} and the recent book of Garcia-Mashreghi-Ross \cite{GMR}). These two monographs contain many references on this rich topic. On the other hand, the M\"untz-Sz\'asz theorem (see \cite{Muntz} for the real case and \cite{Sz} for the complex case, or \cite{Al} for a good survey) gives an answer to a very natural question on the extension of the Weierstrass theorem in approximation theory. More recently, people were interested in another aspect of spaces spanned by monomials in the non dense case: what is their geometry (from a Banach space point of view) and how behave their operators ? One of the main reference on the subject is the monograph \cite{GuLu} for the state of art until '05. See also for instance \cite{AHLM}, \cite{CFT}, \cite{NT} or \cite{GL} for more recent papers.

\medskip

In the sequel, given $a\in\R$, we denote $\C_a=\{z\in\C\,|\; \Re(z)>a\}$, and we shall use the monomials $\e_{\lambda}:t\in(0,1]\mapsto t^\lambda=\e^{\lambda\ln(t)}$ for $\lambda\in\C_{-\frac{1}{2}}$. Clearly, $\e_\lambda$ belongs to $L^2([0,1],dx)$.

\medskip

In this paper, we are interested in a ``dictionary'' between particular model spaces and hilbertian M\"untz spaces. Actually, it turns out that this dictionary appears under a slightely different form in \cite[vol.2, chap.4]{N2}. As a first consequence, as noted in \cite{N2}, this tool allows to recover very quickly the M\"untz-Sz\'asz theorem of density in $L^2([0,1],dx)$ of the space spanned by the monomials $\e_{\lambda_n}$ for a suitable sequence $(\lambda_n)$ of complex numbers. The spirit of this argument is already underlying in a classical proof of the M\"untz theorem for continuous functions, coming back to Feinerman and Newman \cite{FN}, and popularized in the monograph of Rudin \cite{Ru}. Nevertheless, beyond this first application to the  M\"untz-Sz\'asz theorem, we wish to emphasize  the interest of this bridge between two classical areas of function theory, which have many other applications. We give some of them in this paper, but there are many potential others.
\medskip

\begin{definition}
The Hardy space of the right-half plane $\mathcal H^2(\mathbb C_0)$ consists of functions $f$ analytic on 
$\mathbb C_0$ satisfying 
$$
\|f\|_2=\sup_{x>0}\left(\int_{\mathbb R}|f(x+iy)|^2\,\frac{dy}{2\pi}\right)^{1/2}<\infty.
$$
\end{definition}
It is well-known that  $\mathcal H^2(\mathbb C_0)$ can be viewed as a closed subspace of $L^2(i\mathbb R)$ and it is a reproducing kernel Hilbert space whose kernel at point $\lambda\in\mathbb C_0$ is given by 
$$
k_\lambda(z)=\frac{1}{z+\overline{\lambda}},\qquad z\in\mathbb C_0.
$$ 
Given a sequence $\Lambda=(\lambda_n)_{n\geq 1}$ of (distinct) points in $\mathbb C_0$, we recall that the sequence $(k_{\lambda_n})_{n\geq 1}$ is not complete in $\mathcal H^2(\mathbb C_0)$ (which means that it generates a proper closed subspace of $\mathcal H^2(\mathbb C_0)$) if and only if $\Lambda$ satisfies the Blaschke condition, that is  
$$
\sum_{n}\frac{\Re(\lambda_n)}{|\lambda_n|^2+1}<\infty.
$$
In that case, if $B_\Lambda$ denotes the Blaschke product associated to $\Lambda$, we know that the sequence $(k_{\lambda_n})_{n\geq 1}$ is minimal ({\sl i.e.} any vector belongs to the closed subspace generated by the other vectors) and the closed subspace generated by $(k_{\lambda_n})_{n\geq 1}$ is
\begin{equation}\label{eq:completude-nr}
\span(k_{\lambda_n}:n\geq 1)=K_{B_\Lambda}=(B_{\Lambda}\mathcal H^2(\mathbb C_0))^\perp.
\end{equation}

The space $K_{B_\Lambda}$ is a particular case of subspaces $K_\Theta=(\Theta \mathcal H^2(\mathbb C_0))^\perp$, where $\Theta$ is an inner function. These spaces are also called model spaces since their analogue in the Hardy space of the unit disc are involved through the theory of Sz.-Nagy--Foias in the model theory for Hilbert space contractions. Note that $K_\Theta$ (as a closed subspace of $H^2(\mathbb C_0)$) is also a reproducing kernel Hilbert space whose kernel at point $\lambda\in\mathbb C_0$ is given by 
$$
k_\lambda^\Theta(z)=\frac{1-\overline{\Theta(\lambda)}\Theta(z)}{z+\overline{\lambda}},\qquad z\in\mathbb C_0.
$$

As mentioned above, there exists a huge literature on these spaces. We shall only mention here two results concerning the properties of bases of sequences of normalized reproducing kernels, which we will 
use in our paper. Recall that if $(x_n)_{n\geq 1}$ is a minimal and complete sequence of a Hilbert space $\mathcal H$, the sequence $(x_n)_{n\geq 1}$ is called a \emph{Riesz basis} for $\mathcal H$ if there exists two constants $c_1,c_2>0$ such that 
$$
c_1 \left( \sum_{n\geq 1}|a_n|^2\right)^{1/2}\leq \left\|\sum_{n\geq 1}a_n x_n\right\|_{\mathcal H}  \leq c_2 \left( \sum_{n\geq 1}|a_n|^2\right)^{1/2},
$$
for every finitely supported sequence of complex numbers $(a_n)_n$. It is called an \emph{asymptotically orthonormal basis} for $\mathcal H$ if there exists a sequence $(\varepsilon_N)_N$ tending to $0$ and satisfying 
 $$
(1-\varepsilon_N) \left( \sum_{n\geq N\geq 1}|a_n|^2\right)^{1/2}\leq \left\|\sum_{n\geq N}a_n x_n\right\|_{\mathcal H}  \leq (1+\varepsilon_N) \left( \sum_{n\geq N}|a_n|^2\right)^{1/2},
$$
for every finitely supported sequence of complex numbers $(a_n)_n$. 
\begin{theorem}[Volberg \&  Nikolski-Pavlov]\label{NikolskiPavlovVolberg}
Let $\Lambda=(\mu_n)_{n\geq 1}$ be a Blaschke sequence of distinct points of $\mathbb C_0$, let  $B_\Lambda$ be the associated Blaschke product and denote by $x_n=k_{\mu_n}/\|k_{\mu_n}\|_2$, $n\geq 1$ the normalized reproducing kernel.
\begin{enumerate}[(1)]
\item The sequence $(x_n)_{n\geq 1}$ is a Riesz basis for $K_{B_\Lambda}$ if and only if the sequence $\Lambda$ satisfies the so-called Carleson condition, that is
\begin{equation}\label{eq:Carleson-Condition}
\inf_n \prod_{k\neq n}\left|\frac{\mu_n-\mu_k}{\mu_n+\overline{\mu_k}}\right|>0.
\end{equation}
\item The sequence $(x_n)_{n\geq 1}$ is an asymptotically orthonormal basis for $K_{B_\Lambda}$ if and only if the sequence $\Lambda$ is a thin sequence, that is
\begin{equation}\label{eq:thin}
\lim_{n\to \infty} \prod_{k\neq n}\left|\frac{\mu_n-\mu_k}{\mu_n+\overline{\mu_k}}\right|=1.
\end{equation}

\end{enumerate}
\end{theorem}
The part (1) is a result due to Nikolski-Pavlov \cite[p.135]{N}. The second part (2) is due to Volberg \cite{V} (see also a more elementary proof due to Gorkin, McCarthy, Pott and Wick in \cite{GCPW}).

\section{The dictionary}

We consider the following map (Mellin transform):

$${\mathcal M}:\left |\begin{array}{cccl}

 & L^2\big([0,1],\frac{ds}{s}\big) & \longrightarrow & {\mathcal H}^2\big(\C_0\big)\cr
& f &\longmapsto & {\mathcal M}(f)(z)=\dis\int_0^1f(s)s^{z-1}\,ds 
\end{array} \right.$$

The key of our viewpoint is that the map ${\mathcal M}$ is an isometric isomorphism. This is part of folklore and actually a reformulation of the Paley-Wiener theorem (cf \cite[p. 354]{Ru}), but for sake of completeness, we include here the argument. Indeed, for every function $g$ in the Hardy space ${\mathcal H}^2\big(\C_0\big)$, 
thanks to the theorem of Paley-Wiener, there exists a unique function $F$ in $L^2\big(\R^+\big)$ such that 
$$\forall z\in\C_0\, , \qquad g(z)=\int_{\mathbb R^+}F(t)\e^{-tz}\,dt\quad\hbox{and  } \|F\|_{2}=\|g\|_{{\mathcal H}^2({\mathbb\C}_0)}\,.$$

It means that the function $\dis f(s)=F\big(-\ln(s)\big)$ (equivalently $F(t)=f\big(\e^{-t}\big)$) satisfies $g={\mathcal M}(f)$ and 
$$\int_0^1|f(s)|^2\,\frac{ds}{s}=\int_{\mathbb R^+}\big|F(t)\big|^2\,dt=\|g\|^2_{{\mathcal H}^2({\mathbb\C}_0)}\; ,$$
which was our claim.
\medskip

Now, using the fact that the following map $$f\in L^2\big([0,1],ds\big)\longmapsto \sqrt s\cdot  f \in L^2\big([0,1],\frac{ds}{s}\big)$$
is also an isometric isomorphism, we get immediately

\begin{theorem}\label{dico}
The map $$  {\mathcal D}:\left |\begin{array}{cccl}

 & L^2\big([0,1],ds\big) & \longrightarrow & {\mathcal H}^2\big(\C_0\big)\cr
& f &\longmapsto & {\mathcal M}(\sqrt s f)(z)=\dis\int_0^1f(s)s^{z-\frac{1}{2}}\,ds
\end{array} \right.$$

defines an isometric isomorphism.
\end{theorem}

Let us point out that for every $\lambda\in\C_{-\frac{1}{2}}$, we have 
\begin{equation}\label{eq:monome-nr}
\forall z\in\C_0\, ,\qquad {\mathcal D}\big(\e_\lambda\big)(z)=\frac{1}{z+\lambda+\frac{1}{2}}=k_{\bar\lambda+\frac{1}{2}}(z)\; .
\end{equation}

\bigskip

The first immediate application we would like to mention is that we get the classical full M\"untz theorem for (quite) free (see \cite[Ex.4.7.2., vol.2]{N2} too). 

Let $\Lambda=\big(\lambda_n\big)_{n\ge1}\subset\C_{-\frac{1}{2}}$, we denote by $M_\Lambda$ the (vector) space spanned by the $\e_\lambda$ when $\lambda$ runs over $\Lambda$.

\begin{theorem}{(Full M\"untz theorem in $L^2$)}\label{szasz}
Let $\Lambda=(\lambda_n)$ be a sequence of $\C_{-\frac{1}{2}}$. Then

\centerline{$M_\Lambda$ is dense in $L^2\big([0,1],dx\big)$ if and only if $\quad\dis\sum\frac{\frac{1}{2}+\Re(\lambda_n)}{\big|\lambda_n+\frac{1}{2}\big|^2+1}=+\infty$.}

\end{theorem}

\begin{proof} $M_\Lambda$ is dense in $L^2\big([0,1],dx\big)$ if and only if ${\mathcal D}(M_\Lambda)$ is dense in ${\mathcal H}^2\big(\C_0\big)$. But, by \eqref{eq:monome-nr},  ${\mathcal D}(M_\Lambda)$ is the space spanned by the functions $k_{\mu_n}$ where $\mu_n=\overline{\lambda_n}+\frac{1}{2}$ so any function $f$ in its orthogonal space is characterized by
$f\big(\mu_n\big)=0 $ for every $n$. Hence $M_\Lambda$ is dense in $L^2\big([0,1],dx\big)$ if and only if the only possible function $f$ is $f=0$, which happens if and only if the non Blaschke condition $\dis\sum\frac{\Re(\mu_n)}{1+\big|\mu_n\big|^2}=\dis\sum\frac{\frac{1}{2}+\Re(\lambda_n)}{1+\big|\lambda_n+\frac{1}{2}\big|^2}=+\infty$ is satisfied. 
\end{proof}

\bigskip

The main aspect we are interested in now is the non-dense framework. When the Blaschke condition 
\begin{equation}\label{eq:blaschke-condition-trans}
\dis\sum\frac{\frac{1}{2}+\Re(\lambda_n)}{1+\big|\lambda_n+\frac{1}{2}\big|^2}<+\infty
\end{equation}
 is satisfied, we have a proper subspace of  $L^2\big([0,1],dx\big)$, namely
$$M^2_\Lambda=span\big\{\e_{\lambda_n}\,|\;n\ge1\big\}\subsetneq L^2\big([0,1],dx\big).$$

Recently the geometry of such spaces and the  behavior of their operators were studied but actually a lot of natural questions are still open.

From our dictionary, the following result is gained for free but emphasizes a link between theory of M\"untz spaces and the theory of model spaces:

\begin{theorem}\label{dicomuntz}
Let $\Lambda=\big(\lambda_n\big)_{n\ge1}\subset\C_{-\frac{1}{2}}$ be a sequence which satisfies \eqref{eq:blaschke-condition-trans}. Consider $B_\Lambda$ the Blaschke product on $\C_0$ whose zeroes are the $\dis\overline{\lambda_n}+1/2$, for $n\ge1$. Then ${\mathcal D}$ realizes an isometric isomorphism between $M^2_\Lambda$ and the model space $\dis K_{B_\Lambda}$.

\end{theorem}

\begin{proof} Since $\mathcal D$ is an isometry, the result follows immediately from \eqref{eq:completude-nr} and \eqref{eq:monome-nr}.
 \end{proof}

\bigskip

\section{Applications} 

From this dictionary and the theory of model spaces, one can recover some known results and derive new ones. Thanks to this bridge, we can go from the M\"untz side to the model side, or reciprocally.

\subsection{From Model spaces to M\"untz spaces }

One can extend to complex powers the Gurariy-Macaev theorem which was available for real powers only in \cite{GM} in the $L^2$ framework (see \cite[Ex.4.7.2., vol.2]{N2} too):

\begin{corol} Let $\Lambda=\big(\lambda_n\big)_{n\ge1}\subset\C_{-\frac{1}{2}}$ be a sequence which satisfies \eqref{eq:blaschke-condition-trans}. TFAE
\begin{enumerate}
\item $\Big\{\big(2\Re(\lambda_n)+1\big)^\frac{1}{2}\,\e_{\lambda_n}\Big\}$ is a Riesz basis of $M_\Lambda^2$.
\medskip

\item   $\dis\inf_n \prod_{k\ne n}\Big|\frac{\lambda_n-\lambda_k}{\lambda_n+\overline{\lambda_k}+1} \Big| >0.$
\end{enumerate}
\end{corol}

\begin{proof} It follows immediately from Theorem~\ref{dico} and the first part of Theorem~\ref{NikolskiPavlovVolberg} applied to  $\dis\mu_n=\overline{\lambda_n}+\frac{1}{2}\cdot$
\end{proof}
Remarks:
\begin{itemize}
\item It is well known that when the $\lambda_n$ are real and increasing, then the Carleson condition appearing in (2) is equivalent to the fact that $\Lambda$ is lacunary, {\it i.e.} there exists some $c>1$ such that $\lambda_{n+1}\ge c\lambda_n$ for every $n\ge1$.

\item We can also derive some applications in the framework of Dirichlet series: 

\hskip-10pt let $(q_k)$ be an increasing sequence of integers such that ${\dis\inf_n\prod_{k\ne n}\frac{|\ln(q_n/q_k)|}{\ln(q_nq_k\e)} >0}$ then 
$$\dis\int_0^{+\infty}\Big|\sum_{k\ge0} a_kq_k^{-s}\Big|^2\e^{-s}\;ds\approx\sum_{k\ge0}\frac{|a_k|^2}{\ln(q_k)}\,\cdot$$ 
\end{itemize}
\medskip

The Gurariy-Macaev theorem is revisited in \cite{GL} (in $L^p$ spaces) and it is also proved there that, exactly in the case of super-lacunary real sequences, the normalized monomials forms an  asymptotic orthonormal system. The following corollary extends this result to the case of complex powers in the hilbertian framework.

\begin{corol}\label{volb}
Let $\Lambda=\big(\lambda_n\big)_{n\ge1}\subset\C_{-\frac{1}{2}}$ be a sequence which satisfies \eqref{eq:blaschke-condition-trans}. TFAE
\begin{enumerate}
\item For every $a=(a_k)\in\ell^2$, we have
$$\dis\hskip-15pt\big(1-\eps_n\big)\Big(\sum_{k\ge n}|a_k|^2\Big)^\frac{1}{2}\le\Big\|\sum_{k\ge n} a_k\big(2\Re(\lambda_k)+1\big)^\frac{1}{2}\e_{\lambda_k}\Big\|_{L^2}\le\big(1+\eps_n\big)\Big(\sum_{k\ge n}|a_k|^2\Big)^\frac{1}{2}$$
 where $\eps_n\rightarrow0$.
\medskip

\item   $\dis \prod_{k\ne n}\Big|\frac{\lambda_n-\lambda_k}{\lambda_n+\overline{\lambda_k}+1} \Big| \longrightarrow 1\quad,\, \hbox{as } n\rightarrow+\infty.$
\end{enumerate}
\end{corol}

\begin{proof} It follows immediately from Theorem~\ref{dico} and the second part of Theorem~\ref{NikolskiPavlovVolberg} applied to $\dis\mu_n=\overline{\lambda_n}+\frac{1}{2}\cdot$
\end{proof}
\bigskip 

Let $\mu\in\mathbb C_{-\frac{1}{2}}$ and $\Lambda=(\lambda_n)_n\subset\mathbb C_{-\frac{1}{2}}$ be a sequence which satisfies \eqref{eq:blaschke-condition-trans}. Denote by 
$$
x_\mu^\Lambda=P_{M_{\Lambda}^2}(e_\mu),
$$ 
where $P_{M_\Lambda^2}$ denotes the orthogonal projection of $L^2([0,1],dx)$ onto $M_{\Lambda}^2$. Now, given a sequence $(\mu_n)_{n\geq 1}\subset\mathbb C_{-\frac{1}{2}}$, it is natural to study the geometry of sequences $(x^\Lambda_{\mu_n})_n$. What can be said concerning the completeness, the basis properties,...? It turns out that if we combine our dictionary and known results on reproducing kernels of model spaces, we can get several results. We just mention one of them. The key is the following simple lemma.

\begin{lemma}\label{lem-repr-kern-model}
Let $\mu\in\mathbb C_{-\frac{1}{2}}$ and $\Lambda=(\lambda_n)_n\subset\mathbb C_{-\frac{1}{2}}$ be a sequence which satisfies \eqref{eq:blaschke-condition-trans} and let $B_\Lambda$ be the Blaschke product on $\C_0$ whose zeroes are the $\overline{\lambda_n}+\frac{1}{2}$, $n\geq 1$. Then
$$
\mathcal D(x_\mu^\Lambda)=k_{\overline{\mu}+\frac{1}{2}}^{B_\Lambda}.
$$
\end{lemma}

\begin{proof} According to Theorem~\ref{dico},  for every $f\in K_{B_\Lambda}$, there exists a unique $g\in M_\Lambda^2$ such that $f=\mathcal Dg$. On one hand, we have 
$$
\langle f,\mathcal D(x_\mu^\Lambda)\rangle_{\mathcal H^2(\mathbb C_0)}=\langle \mathcal D(g),\mathcal D(P_{M_\Lambda^2}e_\mu)\rangle_{\mathcal H^2(\mathbb C_0)}=\langle g,P_{M_\Lambda^2}e_\mu \rangle_{L^2([0,1]}=\langle g, e_\mu \rangle_{L^2([0,1]}.
$$
On the other hand, we have 
$$
\langle f, k_{\overline{\mu}+\frac{1}{2}}^{B_\Lambda}\rangle_{\mathcal H^2(\mathbb C_0)}=f(\overline{\mu}+\frac{1}{2})=(\mathcal D g)(\overline{\mu}+\frac{1}{2})=\int_{0}^1 g(t)t^{\overline\mu}\,dt=\langle g, e_\mu \rangle_{L^2([0,1]}.
$$
Hence, for every $f\in K_{B_\Lambda}$, we deduce
$$
\langle f,\mathcal D(x_\mu^\Lambda)\rangle_{\mathcal H^2(\mathbb C_0)}=\langle f, k_{\overline{\mu}+\frac{1}{2}}^{B_\Lambda}\rangle_{\mathcal H^2(\mathbb C_0)},
$$
which gives the result. 
\end{proof}
There exists a large literature devoted to geometric properties of sequences of reproducing kernels of model spaces (see for instance \cite{N}). Using Lemma~\ref{lem-repr-kern-model}, we can obtain similar results for sequences $(x_{\mu_n}^\Lambda)_n$. As an example we mention the following. 
\begin{theorem}
Let $\Lambda=(\lambda_n)_n\subset\mathbb C_{-\frac{1}{2}}$ be a sequence which satisfies \eqref{eq:blaschke-condition-trans} and let  $(\mu_n)_n \subset\mathbb C_{-\frac{1}{2}}$ satisfying 
$$
\lim_{n\to\infty}\prod_{k\geq 1}\left|\frac{\mu_n-\lambda_k}{\overline{\mu_n}+\lambda_k+1}\right|=0.
$$
The following are equivalent: 
\begin{enumerate}
\item there exists $N$ sufficiently large such that $(x_{\mu_n}^\Lambda/\|x_{\mu_n}^\Lambda\|)_{n\geq N}$ is a Riesz basis for its closed linear span; 
\item the sequence $(\overline{\mu_n}+\frac{1}{2})_n$ satisfies the Carleson condition \eqref{eq:Carleson-Condition}.
\end{enumerate}

\end{theorem}

\begin{proof} Let $B_\Lambda$ be the Blaschke product associated to $(\overline{\lambda_n}+\frac{1}{2})_n$. According to Lemma~\ref{lem-repr-kern-model}, the sequence $(x_{\mu_n}^\Lambda/\|x_{\mu_n}^\Lambda\|)_{n\geq N}$ is a Riesz basis for its closed linear span if and only the normalized sequence of reproducing kernels $(k_{\overline{\mu_n}+\frac{1}{2}}^{B_\Lambda}/\|k_{\overline{\mu_n}+\frac{1}{2}}^{B_\Lambda}\|)_{n\geq N}$ is a Riesz basis for its closed linear span. The hypothesis means that 
$\lim_{n\to \infty}B_\Lambda(\overline{\mu_n}+\frac{1}{2})=0$ and it remains to apply \cite[Theorem 3.2]{HNP}. 
\end{proof}

In \cite{Baranov}, A. Baranov used an approach of N. Makarov and A. Poltoratski to give a criterion for completeness of systems of reproducing kernels in the model spaces. We can also use our dictionary to get similar results in M\"untz spaces.  In that spirit, we give a stability result for completeness. 

\begin{theorem}\label{Thm-Stability-completeness}
Let $\Lambda=(\lambda_n)_n\subset\mathbb C_{-\frac{1}{2}}$ be a sequence which satisfies \eqref{eq:blaschke-condition-trans} and let  $(\mu_n)_n \subset\mathbb C_{-\frac{1}{2}}$ satisfying $\mathcal R\in L^\infty(\mathbb R)$ where
$$
\mathcal R(t)=\sum_n \left| \frac{\lambda_n-\mu_n}{\mu_n+\frac{1}{2}-it}\right|.
$$
Then $(x_{\mu_n}^\Lambda)_n$ is complete in $M_\Lambda^2$. 
\end{theorem}

\begin{proof}
Let $B_\Lambda$ be the Blaschke product associated to $(\overline{\lambda_n}+\frac{1}{2})_n$. According to Lemma~\ref{lem-repr-kern-model}, the sequence $(x_{\mu_n}^\Lambda)_n$  is complete in $M_\Lambda^2$ if and only if $(k_{\overline{\mu_n}+\frac12}^{B_\Lambda})_n$ is complete in $K_{B_\Lambda}$. It remains to apply the analogue of \cite[Theorem 1.3]{Baranov} in $\mathcal H^2(\mathbb C_0)$.
\end{proof}

Another application concerns the summation basis: given 
$\Lambda=(\lambda_n)_{n\geq 1}\subset\mathbb C_{-\frac{1}{2}}$ satisfying \eqref{eq:blaschke-condition-trans}, we know that $(e_{\lambda_n})_{n\geq 1}$ is minimal. Moreover, when the Carleson's condition is satisfied, we saw previously that  $(e_{\lambda_n})_{n\geq 1}$ is a Riesz basis for $M_{\Lambda}^2$. Actually, in the non dense case, when the Carleson's condition is not satisfied, we can still prove that $(e_{\lambda_n})_{n\geq 1}$ is a summation basis for $M_\Lambda^2$.

Recall that if $(x_n)_{n\geq 1}$ is a complete and minimal sequence in a Hilbert space $H$, and $(x_n^*)_{n\geq 1}$ is its biorthogonal sequence, then $(x_n)_{n\geq 1}$ is said to be a summation basis for the Hilbert space $H$ if there exists an infinite matrix $A=(a_{n,k})_{n,k\geq 1}$ such that for every $x\in H$, we have 
$$
x=\lim_{k\to \infty}\sum_{n=1}^{\infty}a_{n,k}\langle x,x_n^*\rangle x_n.
$$

\begin{theorem}
Let $\Lambda=(\lambda_n)_{n\geq 1}\subset \mathbb C_{-\frac{1}{2}}$ be a sequence which satisfies \eqref{eq:blaschke-condition-trans}. Then $(e_{\lambda_n})_{n\geq 1}$ is a summation basis for $M_{\Lambda}^2$.
\end{theorem}

\begin{proof}
Let $B$ be the Blaschke product whose zeroes are $\mu_n=\overline{\lambda_n}+\frac{1}{2}$, $n\geq 1$. According to Theorem~\ref{dicomuntz}, the operator $\mathcal D$ realizes an isometric isomorphism between $M_\Lambda^2$ and $K_B$ and 
$\mathcal D(e_{\lambda_n})=k_{\mu_n}$, $n\geq 1$, where $k_\mu$ is the reproducing kernel of ${\mathcal H}^2\big(\C_0\big)$. Denote by 
$$
B^{(k)}=\prod_{n\geq k}\alpha_nb_{\mu_n},\qquad k\geq 1,
$$
where $b_\mu(z)=(z-\mu)/(z+\overline{\mu})$ is the elementary Blaschke factor in ${\mathcal H}^2\big(\C_0\big)$ and $\alpha_n$ is a suitable complex number with modulus $1$. It is known that for every $g\in K_B$, we have 
$$
g=\lim_{k\to\infty}\sum_{n\geq 1}\overline{B^{(k)}(\mu_n)}\langle g,k_{\mu_n}^*\rangle k_{\mu_n},
$$
see \cite[page 194]{N} or \cite[vol.1., p. 620, Ex.15.3.1.]{FM}. Using the isomorphism $\mathcal D$, we get that for every $f\in M_\Lambda^2$, 
$$
\begin{aligned}
f&=\lim_{k\to \infty} \sum_{n\geq 1}\overline{B^{(k)}(\mu_n)} \langle \mathcal D f,k_{\mu_n}^*\rangle e_{\lambda_n} \\
&=\lim_{k\to \infty} \sum_{n\geq 1}\overline{B^{(k)}(\mu_n)} \langle f, \mathcal D ^*k_{\mu_n}^*\rangle e_{\lambda_n}
\end{aligned}
$$
Note that 
$$
\langle e_{\lambda_p},\mathcal D^*k_{\mu_n}^*\rangle=\langle \mathcal D(e_{\lambda_p}),k_{\mu_n}^*\rangle =\langle k_{\mu_p},k_{\mu_n}^*\rangle=\delta_{n,p}, 
$$
which gives that $\mathcal D^*(k_{\mu_n}^*)=e_{\lambda_n}^*$, $n\geq 1$. Hence
$$
f=\lim_{k\to \infty} \sum_{n\geq 1}\overline{B^{(k)}(\mu_n)} \langle f, e_{\lambda_n}^*\rangle e_{\lambda_n}.
$$
That proves that $(e_{\lambda_n})_{n\geq 1}$ is a summation basis for $M_\Lambda^2$. 
\end{proof}

We immediately get the following. 
\begin{corol}
Let $\Lambda=(\lambda_n)_{n\geq 1}\subset \mathbb C_{-\frac{1}{2}}$ be a sequence which satisfies \eqref{eq:blaschke-condition-trans}. If $(e_{\lambda_n}^*)_{n\geq 1}$ is the biorthogonal sequence associated to $(e_{\lambda_n})_{n\geq 1}$. Then $(e_{\lambda_n}^*)_{n\geq 1}$ is complete in $M_\Lambda^2$. 
\end{corol}

\subsection{From M\"untz spaces to model spaces}

In the following, we revisit the known inequalities of Markov-Newman type to get some new ones in the framework of model spaces. In this spirit, there are many of them but we choose to  mention the following immediate consequence of Theorem 3.4. of \cite{BEZ}.

\begin{theorem}
Let  $\dis(w_n)_{n\geq 1}\subset \mathbb C_{0}$.

Then, for every finite sequence of complex numbers $(a_k)_{1\le k\le n}$
$$\Big\|\sum_{k=1}^n a_k \frac{w_k-\frac{1}{2}}{z+w_k} \Big\|_{{\mathcal H}^2}\le \Bigg[ \dis\sum_{k=1}^n\big|w_k-\frac{1}{2}\big|^2+\sum_{k=1}^n\Big(\Re(w_k)  \sum_{j=k+1}^n\Re(w_j)\Big) \Bigg]^\frac{1}{2} \Big\|\sum_{k=1}^n a_k  \frac{1}{z+w_k}  \Big\|_{{\mathcal H}^2}$$

\end{theorem}

Let us mention a particular  consequence, which is also a simple traduction of the Markov-Newman inequality: assume that the $w_k\ge\frac{1}{2}$ are real numbers, then, for every finite sequence of complex numbers $(a_k)_{1\le k\le n}$:

$$\Big\|\sum_{k=1}^n a_k\frac{w_k-\frac{1}{2}}{z+w_k} \Big\|_{{\mathcal H}^2}\le\sqrt2\Big[ \dis\sum_{k=1}^n w_k\Big] \Big\|\sum_{k=1}^n\frac{a_k }{z+w_k}\Big\|_{{\mathcal H}^2}$$

\bigskip

\noindent{\it A short and elementary proof of Volberg's theorem in the real case}. 
\medskip

In the particular case of real exponents, the proof of Corollary \ref{volb} proposed in \cite{GL} is elementary. We wish to propose here a new proof of the difficult result of Volberg in this particular case. The argument below does not use our dictionary,  nevertheless it clearly follows from the spirit of this paper: it is the simple traduction of the proof for Müntz spaces in \cite{GL} to model spaces.

The statement in this case reads as follows: the normalized reproducing kernels associated to an increasing sequence of positive real numbers $(w_n)$ forms an asymptotic orthonormal system if and only if $\dis \frac{w_{n+1}}{w_n}\rightarrow+\infty$.\medskip

Let us fix a super-lacunary sequence of real numbers $(w_n)$, which means  that $\dis \frac{w_{n+1}}{w_n}\rightarrow+\infty$. Equivalently $q_n=\big(2w_n\big)^\frac{1}{2}$ is super-lacunary.

Now take any finitely supported sequence of scalars $(a_k)$ and develop
$$\begin{array}{ccl}
\dis\Big\|\sum_{k\ge n} a_k \frac{\big(2w_k \big)^\frac{1}{2}}{z+w_k} \Big\|^2_{{\mathcal H}^2}&=&\dis\sum_{k,l\ge n} a_k\overline{ a_l}q_kq_l\big\langle\frac{1}{z+w_k},\frac{1}{z+w_l}\big\rangle\cr
&=&\dis\sum_{k,l\ge n} 2a_k\overline{ a_l}\frac{q_kq_l}{q_k^2+q_l^2}\cr 
&=&\dis\|a\|_2^2+\biindice{\sum}{k,l\ge n}{k\neq l} 2a_k\overline{ a_l}\frac{q_kq_l}{q_k^2+q_l^2}\qquad(\ast)\cr 

\end{array}
$$

But $\dis \big| 2a_k\overline{ a_l}\big|\le |a_k|^2+|a_l|^2$ so that 
$$\Bigg|\biindice{\sum}{k,l\ge n}{k\neq l} 2a_k\overline{ a_l}\frac{q_kq_l}{q_k^2+q_l^2}\Bigg|\le2\|a\|_2^2\sup_{l\ge n}\dis\biindice{\sum}{k\ge n}{k\neq l}\frac{q_kq_l}{q_k^2+q_l^2}\,\cdot$$
Since for every $l\ge n$, we have 
$$\biindice{\sum}{k\ge n}{k\neq l}\frac{q_kq_l}{q_k^2+q_l^2}=2\sum_{k>l\ge n}\frac{q_kq_l}{q_k^2+q_l^2}\le2\sum_{k>l\ge n}\frac{q_l}{q_k}\le2\frac{1}{r_n-1}\; , $$
where $\dis r_n=\inf_{m\ge n}\frac{q_{m+1}}{q_m}\rightarrow\infty$.

Hence $$\Bigg|\biindice{\sum}{k,l\ge n}{k\neq l} 2a_k\overline{ a_l}\frac{q_kq_l}{q_k^2+q_l^2}\Bigg|\le4\frac{1}{r_n-1}\|a\|_2^2\,\cdot\qquad(\ast\ast)$$

Writing $\dis1+\eps_n=\Big(1+4\frac{1}{r_n-1}\Big)^\frac{1}{2}$, we get the majorization:

$$\dis\Big\|\sum_{k\ge n} a_k \frac{\big(2w_k \big)^\frac{1}{2}}{z+w_k} \Big\|_{{\mathcal H}^2}\le(1+\eps_n)\|a\|_2.$$

But using again $(\ast)$ and $(\ast\ast)$, we have

$$\Big\|\sum_{k\ge n} a_k \frac{\big(2w_k \big)^\frac{1}{2}}{z+w_k} \Big\|^2_{{\mathcal H}^2}\ge\Big(1-4\frac{1}{r_n-1}\Big)\|a\|_2^2$$

and we get $\dis\inf_{\|a\|_2=1}\Big\|\sum_{k\ge n} a_k\frac{\big(2w_k \big)^\frac{1}{2}}{z+w_k}  \Big\|_{{\mathcal H}^2}\ge(1-2\eps_n-\eps_n^2)^\frac{1}{2}\longrightarrow1\, .$

The "if" part of the statement follows.
\medskip

The necessary condition is easy to get:  for every $t>0$, we have

$$(1+\eps_n)(1+t^2)^\frac{1}{2}\ge\Big\|\frac{\big(2w_n \big)^\frac{1}{2}}{z+w_n} + t\frac{\big(2w_{n+1} \big)^\frac{1}{2}}{z+w_{n+1}} \Big\|_{{\mathcal H}^2}\ge\Big\langle\frac{\big(2w_n \big)^\frac{1}{2}}{z+w_n} + t\frac{\big(2w_{n+1} \big)^\frac{1}{2}}{z+w_{n+1}}, \frac{\big(2w_n \big)^\frac{1}{2}}{z+w_n}  \Big\rangle  $$
which gives $$\dis(1+\eps_n)(1+t^2)^\frac{1}{2}\ge 1+2t \frac{\sqrt{\frac{w_{n+1}}{w_n}}}{1+\frac{w_{n+1}}{w_n}}$$
and this forces  $\dis \frac{w_{n+1}}{w_n}\rightarrow+\infty$.
\hfill $\square$

Actually, the same results holds (with quite the same proof) if we assume that  $(w_n)$ is a decreasing sequence of positive real numbers: in that case, the condition reads as $\dis \frac{w_{n+1}}{w_n}\rightarrow0$.\medskip
\bigskip


\end{document}